\documentclass[11pt,leqno, a4paper]{amsart}

\usepackage[margin=1.5in]{geometry}
\usepackage{graphicx}
\usepackage{amssymb}
\usepackage{color}
\usepackage{setspace}
\usepackage{mathtools}
\usepackage{dsfont}
\usepackage{commath}
\usepackage{enumerate}
\usepackage{mathrsfs}
\usepackage{fancyhdr}
\usepackage{amsthm}
\usepackage{scalerel} 
\usepackage{fancyhdr}
\usepackage{bbm}
\usepackage{hyperref}
\usepackage[normalem]{ulem}
\usepackage{tikz}
\usepackage{subcaption}
\usepackage{cleveref}
\crefdefaultlabelformat{#2{\scshape #1}#3}
\crefname{subfigure}{Figure}{figs.}
\Crefname{subfigure}{Figure}{Figs.}

\usetikzlibrary{decorations.fractals}

\newtheorem{theorem}{Theorem}[section]

\theoremstyle{definition}
\newtheorem{definition}[theorem]{Definition}

\theoremstyle{remark}
\newtheorem{remark}[theorem]{Remark}

\numberwithin{equation}{section}

\newcommand{\rn}{\mathbb{R}^n}

\newcommand{\cH}{\mathcal{H}}

\newcommand{\St}{\mathbb{S}}
\newcommand{\R}{\mathbb{R}}

\newcommand{\spt}{{\rm{spt }}}

\newcommand{\dist}{\operatorname{dist}}

\newcommand\subsetsim{\mathrel{%
  \ooalign{\raise0.2ex\hbox{$\subset$}\cr\hidewidth\raise-0.8ex\hbox{\scalebox{0.9}{$\sim$}}\hidewidth\cr}}}

\newcommand{\defeq}{\vcentcolon=}

\def\XXint#1#2#3{{\setbox0=\hbox{$#1{#2#3}{\int}$ }
\vcenter{\hbox{$#2#3$ }}\kern-.6\wd0}}

\newcommand{\restr}{\mathbin{\vrule height 1.6ex depth 0pt width
0.13ex \vrule height 0.13ex depth 0pt width 1.3ex}}

\title[Anisotropic regularity in $\R^{2}$]{A geometric proof of regularity of all anisotropic minimal surfaces in $\R^{2}$}
\author{Max Goering}
\thanks{During the preparation of this note, the author was partially supported by FRG DMS-1853993}

\begin{document}
\maketitle

\begin{abstract}
A set of locally finite perimeter $E \subset \R^{n}$ is called an anisotropic minimal surface in an open set $A$ if $\Phi(E;A) \le \Phi(F;A)$ for some surface energy $\Phi(E;A) = \int_{\partial^{*}E \cap A} \| \nu_{E}\| d \cH^{n-1}$ and all sets of locally finite perimeter $F$ such that $E \Delta F \subset \subset A$.

In this short note we provide the details of a geometric proof verifying that all anisotropic surface minimizers in $\R^{2}$ whose corresponding integrand $\| \cdot \|$ is strictly convex are locally disjoint unions of line segments. This demonstrates that, in the plane, strict convexity of $\| \cdot \|$ is both necessary and sufficient for regularity. The corresponding Bernstein theorem is also proven: global anisotropic minimizers $E \subset \R^{2}$ are half-spaces.
\end{abstract}

\section{Introduction}
After De Giorgi's pioneering work on the regularity of area minimizing surfaces which arise as boundaries to sets of locally finite perimeter, much interest has arisen when replacing ``area'' with ``anisotropic energies'' of the form \eqref{e:sv1}.

It is well-known that strict convexity of the integrand $\| \cdot \|$ is necessary for there to be a robust regularity theory, see for instance \cite[Remark 20.4]{maggi2012sets}. It is also known that creating competitors by intersecting with half-spaces can only reduce the energy, see for instance \cite[Remark 20.3]{maggi2012sets}.  Focusing our attention on $1$-dimensional boundaries in $\R^{2}$ we show that strict convexity is not only necessary, but also sufficient for a robust regularity result, Theorem \ref{t:1}. The heart of the proof boils down to a localized version of the fact that intersections with half-spaces reduce energy. 

The technique used to prove Theorem \ref{t:1} fails in higher-dimensions because of the potential existence of saddle points. At a saddle point, one cannot create a competitor by this localization argument. This observation could be thought of as a qualitative version of, or just motivation to defend, the statement that (anisotropic) minimal surfaces have (anisotropic) mean curvature zero.

\section{Preliminaries}

The notation used, and presentation of this section is heavily influenced by \cite{maggi2012sets}.

Suppose $\| \cdot \| : \St^{1} \to (0, \infty)$ is a measurable function. We say such a function $\| \cdot \|$ is strictly convex if its $1$-homogeneous extension to $\R^{2} \setminus \{0\}$ is strictly convex. 

Corresponding to a given convex function $\| \cdot \|$ and some open set $A_{0} \subset \R^{2}$ with finite perimeter, we consider the functional
\begin{equation} \label{e:sv1}
\Phi(E;A_{0}) \defeq \int_{\partial^{*}E \cap A_{0}} \| \nu_{E}\| d \cH^{1}.
\end{equation} 
\begin{definition} \label{d:minimizes}
For a set of locally finite perimeter $A_{0}$ and a mapping $\| \cdot \| : \St^{1} \to (0, \infty)$, we say that a set of locally finite perimeter $E$ minimizes $\Phi( \cdot \, ; A_{0})$ if $\partial E = \spt \mu_{E}$ and for all sets of locally finite perimeter $F$ such that $\overline{E \Delta F} \subset \subset A_{0}$ it holds that
\begin{equation*}
\Phi(E;U) \le \Phi(F; U),
\end{equation*}
where $U \supset \overline{E \Delta F}$ is a pre-compact, open subset of $A_{0}$. 
\end{definition}

The purpose of the set $A_{0}$ in \Cref{d:minimizes} is to define boundary condititions. See \cref{f:1}.

\begin{remark}
The requirement that $\partial E = \spt \mu_{E}$ is necessary in order to be able to make topological claims about the boundary of an anisotropic minimizer. Fortunately, given any set of locally finite perimeter $E$, there exists some borel set $E^{\prime}$ so that $\spt \mu_{E^{\prime}} = \partial E^{\prime}$. See, for instance, \cite[Remark 16.11]{maggi2012sets}. Therefore, this requirement boils down to choosing the ``correct representative'' of $E$ among all equivalent sets of locally finite perimeter.
\end{remark}

\begin{figure*}
	\centering
	\begin{subfigure}[b]{0.48\textwidth}
		\includegraphics[width = \textwidth]{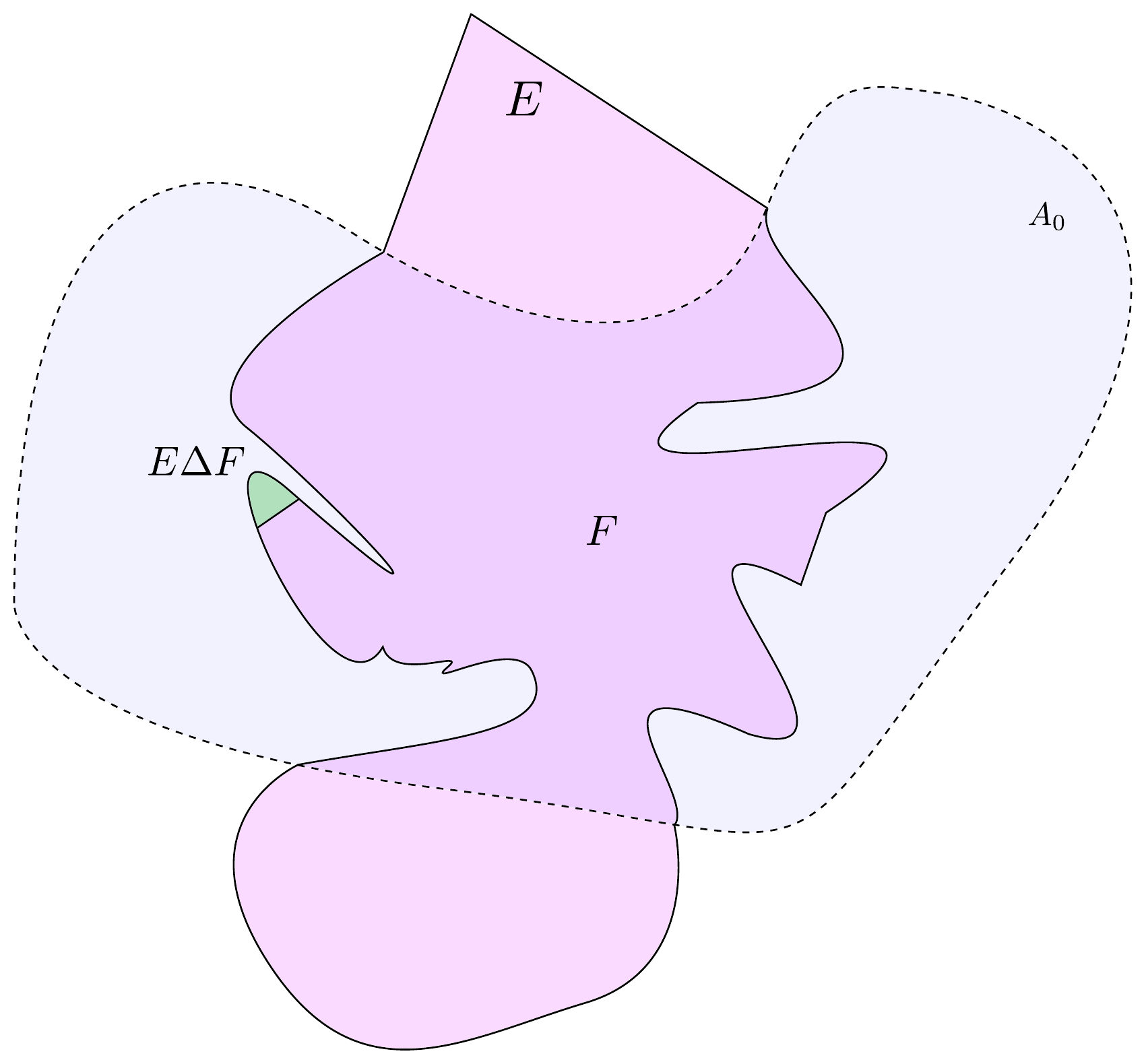}
		\caption{A valid competitor $F$ relative to the set $A_{0}$.}
		\label{f:1}
	\end{subfigure}
	\quad
	\begin{subfigure}[b]{0.48\textwidth}
	\includegraphics[width = \textwidth]{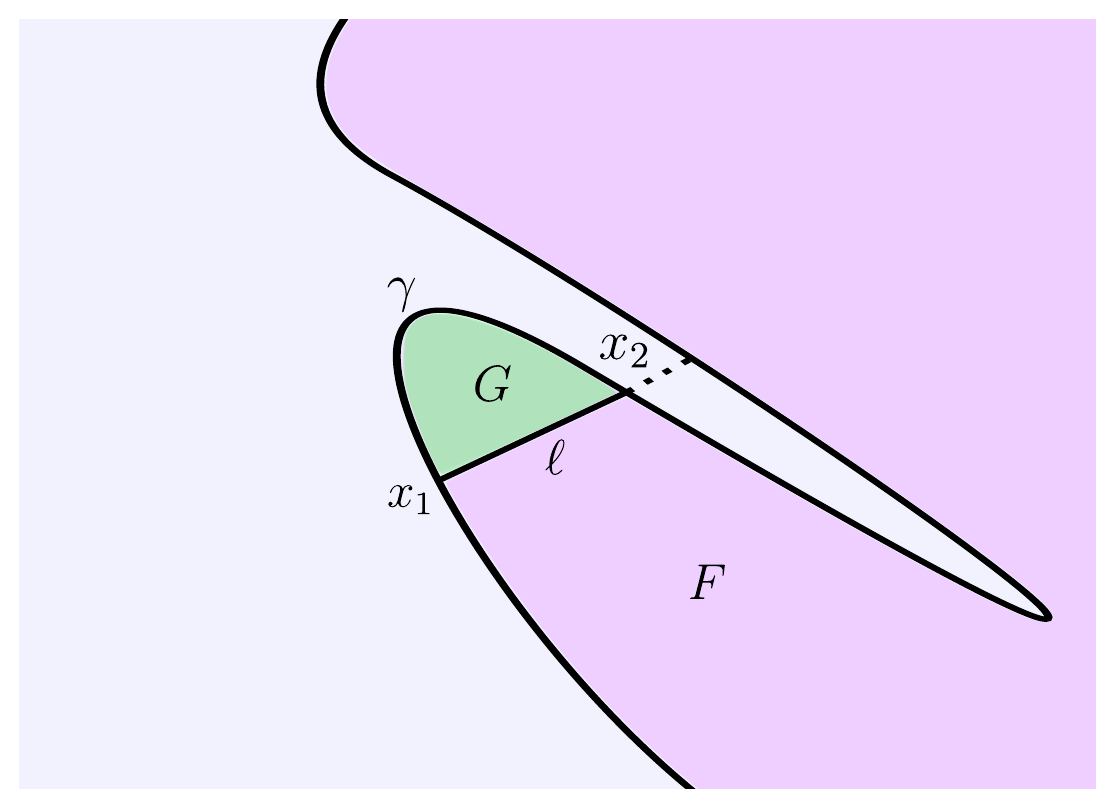}
	\caption{\label{f:2} If $\gamma \cap \ell \neq \{x_{1}, x_{2}\}$ shorten $\ell$ by removing the dashed line segment and redefining $\gamma$ accordingly.}
	\end{subfigure}
\end{figure*}

\begin{remark} \label{r:nocrossings}
If $E \subset \R^{2}$ is $\Phi( \cdot \, ; A_{0})$ minimizing, then $\partial E \cap A_{0}$ contains no self-crossings, or else one could reduce the energy $\Phi$ by removing the loop formed by $\partial E$ crossing itself. 
\end{remark}

We follow the convention that if $A, B \subset \R^{2  }$ then $A \approx B$ means $\cH^{1}(A \Delta B) = 0$, and $A \subsetsim B$ means $\cH^{1} \left( A \setminus B \right) = 0$. Moreover, when considering a set of locally finite perimeter $A$ we will always work with a representation of $A$ so that $\partial A = \spt \mu_{A}$.

For a set of locally finite perimeter $A$, let $\mu_{A}$ denote the Gauss-Green measure associated to $A$, $\nu_{A}$ denote the outward pointing measure theoretic normal, and $\partial^{*}A$ denote the reduced boundary of $A$.

Given a set $A \subset \R^{2}$ and a number $s \in [0,1]$ define
$$
A^{(s)} = \left\{ x \in \R^{2} : \lim_{r \downarrow 0} \frac{ \cH^{2} \left( A \cap B(x,r) \right)}{ \cH^{2}(B(x,r))} = s \right\}.
$$
For a set of locally finite perimeter $A \subset \R^{2}$ the essential boundary of $A$, denoted $\partial^{e}A$ is defined to be the set $\R^{2} \setminus \left( E^{(0)} \cup E^{(1)} \right)$.

We now recall a technical lemma due to Federer.

\begin{theorem}[Federer's theorem] \label{t:federer} If $E$ is a set of locally finite perimeter in $\rn$, then $\partial^{*}E \subset E^{(1/2)} \subset \partial^{e}E$, and
$$
\cH^{n-1}(\partial^{e}E \setminus \partial^{*}E) = 0.
$$
In particular, for any Borel set $M \subset \rn$,
$$
M \approx \left( M \cap E^{(1)} \right) \cup \left( M \cap E^{(0)} \right) \cup \left( M \cap \partial^{*}E \right).
$$
\end{theorem}

We also recall the effect that some set operations have on Gauss-Green measures and reduced boundaries
\begin{theorem}[Set operations on Gauss-Green measures] \label{t:soggm}
If $E$ and $F$ are sets of locally finite perimeter, then

\begin{equation}\label{e:soggm}
\mu_{E \setminus F}  = \mu_{E} \restr F^{(0)} - \mu_{F} \restr E^{(1)} + \nu_{E} \cH^{n-1} \restr \{ \nu_{E} = - \nu_{F}\}.
\end{equation}
and
\begin{equation} \label{e:sord}
\partial^{*} \left( E \cup F \right) \approx \left( F^{(0)} \cap \partial^{*} E \right) \cup \left( E^{(0)} \cap \partial^{*}F \right) \cup \left\{ \nu_{E} = \nu_{F} \right\} 
\end{equation}

\end{theorem}

\section{The Regularity Theorem}
Our main goal is to prove the following theorem.

\begin{theorem} \label{t:1}
Suppose $\| \cdot \| : \St^{1} \to (0, \infty)$ is a lower semicontinuous, bounded, strictly convex function and $A_{0} \subset \R^{2}$ is an open set of locally finite perimeter. Then there exists a $\Phi( \cdot \, ; A_{0})$ minimizer which we denote by $E$. 

Moreover, if $E$ minimizes $\Phi(\cdot  \, ; A_{0})$  then there exists a set equivalent to our minimize, which we also call $E$,so that whenever $\partial E \cap A_{0} \neq \emptyset$ it follows $\partial E \cap A_{0}$ is a non-intersecting collection of line segments. In the case that $A_{0} = \R^{2}$, $E$ must be a half-space.
\end{theorem}

\begin{remark}
The existence portion of Theorem \ref{t:1} is well-known. See, for instance \cite[Remark 20.5]{maggi2012sets} and the historical notes and citations therein. 
\end{remark}

We reiterate that the geometric idea behind the of proof of Theorem \ref{t:1} is known and can even be seen in Federer's definition of an elliptic integrand. The technicalities that arise are primarily due to showing that a point where the boundary is not flat ensures a localized version of the half-plane argument from, for instance \cite[Remark 20.3]{maggi2012sets}, creates a valid competitor.

We first make use of the semicontinuity and boundedness of $\| \cdot \|$ to make a substantial simplification.

\begin{remark}[$\partial E$ is locally Lipschitz for anisotropic minimal surfaces] \label{r:lipschitz}
Let $\| \cdot \|$ be as in the statement of Theorem \ref{t:1}. Since $\| \cdot \|$ is a positive lower semicontinuous function on $\St^{1}$, it achieves a minimum. Since it is also bounded this means there exist $c, C > 0$ such that $c | \nu | \le \| \nu \| \le C |\nu|$ for all $\nu \in \R^{2} \setminus \{0\}$. By a standard competitor argument which requires building competitors by removing balls and the differential inequality afforded by the isoperimetric inequality,\footnote{For more details see the proof of, for instance, \cite[Theorem 21.11]{maggi2012sets}} this implies that if $E$ minimizes $\Phi( \cdot, A_{0})$ and $x \in \partial^{*}E \cap A_{0}$ then there exists $C_{A}=C_{A}(c,C)$ independent of $x$ such that for all $r \in (0, \dist(x, \partial A_{0}))$,
$$
C_{A}^{-1} \le \frac{ \cH^{1}( \partial^{*}E \cap B(x,r))}{r} \le C_{A}.
$$ 
That is, $|\mu_{E}|$ is Ahlfors regular at small, but locally uniform, scales for points $x \in \partial^{*}E$. This has two immediate consequences: the lower bound ensures that there are no isolated points in $\partial E$. The upper-bound guarantees that $\spt \mu_{E} = \overline{ \partial^{*} E}$. It follows from our representation of $E$ that 
\begin{equation} \label{e:minbound}
\cH^{1}( (\partial E \setminus \partial^{*}E )\cap A_{0}) = 0.
\end{equation}

In particular, if $K$ is a compact subset of $A_{0}$, Wa\.zewski's theorem ensures that each connected component of $\partial E \cap K$ is a Lipschitz curve since $\cH^{1}(K \cap \partial E) < \infty$ and $K \cap \partial E$ is compact. In particular, connected components of $\partial E \cap A_{0}$ are locally Lipschitz curves.
\end{remark}

\begin{theorem} \label{t:2}
If $\| \cdot \| : \St^{1} \to (0, \infty)$ is a lower semicontinuous, bounded, strictly convex function, $A_{0} \subset \R^{2}$ is an open set, and $E \subset \R^{2}$ minimizes $\Phi( \cdot \, ; A_{0})$ then there exists an equivalent set of locally finite perimeter which we also call $E$, so that $\partial E \cap A_{0} \neq \emptyset$ implies $\partial E \cap A_{0}$ is a collection of non-intersecting line segments. In the case that $A_{0} = \R^{2}$, $E$ must be a half-space.
\end{theorem}

\begin{proof}
Without loss of generality, assume $E = E^{(1)}$. Suppose for the sake of contradiction that $\partial E \cap A_{0} \neq \emptyset$ is not made up of exclusively straight, non-intersecting line segments.

Then, there exists a non-flat curve $\gamma \subset \partial E$ such that the endpoints of $\gamma$, denoted by $\{x_{1}, x_{2}\}$, satisfy 
\begin{equation} \label{e:1d1}
|x_{1} - x_{2}| < \dist(\gamma, \partial A_{0}).
\end{equation}
By Remark \ref{r:nocrossings}, $\gamma$ has no self-crossings nor does it cross $\partial E \setminus \gamma$.

Let $\ell$ be the line segment between $x_{1}$ and $x_{2}$. If $x \in \ell$ then in light of \eqref{e:1d1}
$$
\dist(x, \gamma) \le \frac{1}{2} \dist(x, \{x_{1}, x_{2}\}) < \dist(\gamma, \partial A_{0})
$$ 
Which verifies $\ell \subset \subset A_{0}$ and consequently, $\ell \cup \gamma \subset \subset A_{0}$. If necessary, shorten $\ell$ (and then $\gamma$ accordingly) so that $\ell \cap \partial E = \gamma \cap \ell = \{x_{1}, x_{2}\}$. The fact that ``the next crossing'' of $\ell$ with $\partial E$ exists follows from Remark \ref{r:lipschitz}. 

In particular, $\gamma \cup \ell$ is a Jordan curve. Since $\ell \cup \gamma \subset \subset A_{0}$, this ensures there exists a unique connected component $G$ of $A_{0} \setminus (\gamma \cup \ell)$ whose closure does not meet $\partial A_{0}$. See \cref{f:2}. 

At this point there are two cases to consider: when $G \subset E$ and when $G \subset E^{c}$. \footnote{If $\| \cdot \|$ were such that $\| x\| =  \| - x\|$ for all $x \in \R^{2} \setminus \{0\}$ one could just replace $E$ with $E^{c}$ to cover both cases simultaneously. However, this additional assumption on $\| \cdot \|$ is not necessary.} 

First consider the case where $G \subset E$. Define the competitor
$F = E \setminus G$. By choice of $G$, $E \Delta F \subset \subset A_{0}$. So that $F$ is a valid competitor for $E$ in $A_{0}$. 

Moreover, $F \subset E$ ensures $\{ \nu_{E} = - \nu_{F}\} = \emptyset$. Hence, \eqref{e:soggm} implies that $G$ satisfies
\begin{equation} \label{e:1d10}
\mu_{G} = \mu_{E \setminus F} = \mu_{E}\restr F^{(0)} - \mu_{F} \restr E^{(1)}.
\end{equation}

Since $F^{(1)} \subset E^{(1)}$ is disjoint from $E^{(1/2)} \supset \partial^{*}E $ we have $\mu_{E} \restr F^{(0)} = \mu_{E} \restr \left( F^{(0)} \cup F^{(1)} \right) $. Since $\cH^{1} \left( \R^{2} \setminus \left( F^{(0)} \cup F^{(1)} \cup \partial^{*}F \right) \right) = 0$ and $|\mu_{E}|$ is absolutely continuous with respect to $\cH^{1} \restr \partial^{*} E$, this in turn implies
\begin{equation} \label{e:1d11}
\mu_{E} \restr F^{(0)} = \mu_{E} \restr \left( F^{(0)} \cup F^{(1)} \right) = \mu_{E}\restr \left(\partial^{*}E \setminus \partial^{*}F\right).
\end{equation}
Similarly
\begin{equation} \label{e:1d12}
\mu_{F} \restr E^{(1)} = \mu_{F} \restr \left( \partial^{*}F \setminus \partial^{*} E \right).
\end{equation}

Combining \eqref{e:1d10}, \eqref{e:1d11}, and \eqref{e:1d12} yields
\begin{align} \label{e:1d13}
\mu_{G} &= \mu_{E} \restr (\partial^{*}E \setminus \partial^{*}F) - \mu_{F} \restr (\partial^{*}F \setminus \partial^{*}E).
\end{align}

Next, we aim to show geometrically evident fact (see \cref{f:2}) that
\begin{equation} \label{e:1d131}
\mu_{G} = \mu_{E} \restr \gamma - \mu_{F} \restr \ell.
\end{equation}

To this end, first note that Remark \ref{r:lipschitz} ensures $\partial E \approx \partial^{*}E$. But, since $\partial F \subset \left( \ell \cup \partial E \right)$ and $\cH^{1}\left(\ell \cap \partial E\right) = 0$, it follows from the flatness of $\ell$ that $\partial F \approx \partial^{*} F$. Similarly, $\partial^{*}G  \approx \partial G$. By Federer's theorem and \eqref{e:minbound} this also implies, $G^{(0)} \approx \R^{2} \setminus \overline{G}$.

Therefore, since $G = E \Delta F$ implies $\partial E \setminus \overline{G} = \partial F \setminus \overline{G}$, it follows
\begin{equation} \label{e:1d003}
 G^{(0)} \cap \partial^{*}F \approx (\R^{2} \setminus \overline{G} ) \cap \partial F = (\R^{2} \setminus \overline{G}) \cap \partial E \approx G^{(0)} \cap \partial^{*}E .
\end{equation}

Moreover, $F \cap G = \emptyset$ implies $\{ \nu_{F} = \nu_{G}\} = \emptyset$ so that \eqref{e:sord} implies
\begin{equation} \label{e:1d00}
\partial^{*}E = \partial^{*} \left(F \cup G \right) \approx \left( F^{(0)} \cap \partial^{*}G \right) \cup \left( G^{(0)} \cap \partial^{*}F \right) .
\end{equation}
Similarly,
\begin{equation} \label{e:1d002}
\partial^{*}F = \partial^{*}(E \setminus G) \approx \left(E^{(1)} \cap \partial^{*}G\right) \cup \left(G^{(0)} \cap \partial^{*} E\right).
\end{equation}
However, since $\partial^{*}E \cap E^{(1)} = \emptyset$ and $\partial^{*}F \cap F^{(0)} = \emptyset$, \eqref{e:1d003} \eqref{e:1d00} and \eqref{e:1d002} imply
$$
\begin{cases} 
\partial^{*}E \setminus \partial^{*}F \approx \left(F^{(0)} \cap \partial^{*}G \right) \\
\partial^{*}F \setminus \partial^{*}E \approx \left(E^{(1)} \cap \partial^{*}G \right).
\end{cases}
$$
Since $\partial G = \gamma \cup \ell$ with $\gamma \subsetsim F^{(0)}$, $\ell \subsetsim E^{(1)}$ and $\ell \cap \gamma \approx \emptyset$, this verifies \eqref{e:1d131}.

Since $G \subset \subset A_{0}$, it follows $\mu_{G}(A_{0}) = 0$. Indeed, choose $\varphi \in C_{c}^{1}(A_{0})$ such that $\varphi \equiv 1$ on $\overline{G} \supset \spt \mu_{G}$ and observe
\begin{equation} \label{e:0}
\mu_{G}(A_{0}) = \int_{A_{0}} \varphi d \mu_{G} = \int_{A_{0}} \nabla \varphi dx = 0.
\end{equation}
Combining \eqref{e:1d131}, \eqref{e:0},  and the fact that $\nu_{F} \restr \ell$ is constant yields
\begin{equation} \label{e:1d14}
 \int_{\ell} \| \nu_{F}\| d \cH^{1} = \left\| \int_{\ell} \nu_{F} d \cH^{1} \right\| = \left\| \int_{\gamma } \nu_{E} d \cH^{1} \right\|  ,
\end{equation}
where we identify $\| \cdot \|$ with its $1$-homogeneous extension. Since $\| \cdot \|$ is strictly convex and $\gamma$ is not flat (so $\nu_{E} \restr \gamma$ is not constant) we further have 
\begin{equation} \label{e:1d15}
\left\| \int_{\gamma} \nu_{E} d \cH^{1} \right\| < \int_{\gamma} \| \nu_{E}\| d \cH^{1}.
\end{equation}

It now follows from \eqref{e:1d13}, \eqref{e:1d131}, \eqref{e:1d14}, and \eqref{e:1d15} that $\Phi(E;A_{0}) > \Phi(F;A_{0})$. Since $F$ is a valid competitor, this contradicts the $\Phi( \cdot \, ; A_{0})$ minimality of $E$, completing Case 1.

In case $G \subset E^{c}$ define $F = E \cup G$. Since $G = E \Delta F$, 
is compactly contained in $A_{0}$, this case follows analogously to previous one.

It remains to show that if $A_{0} = \R^{2}$ then $E$ is a half-space. Indeed, we know that $\partial E$ must be a collection of non-intersecting lines, and if $\partial E$ contains more than one line, they must be parallel. Let $L_{1}, L_{2}$ be two consecutive lines in $\partial E$. Let $\vec{s}$ be a unit vector parallel to $L_{1}$ and $\vec{t}$ be orthonormal to $\vec{s}$. 

The idea is is to build a competitor $F$ whose boundary is identical to $\partial E$, except on some rectangle, where on this rectangle, the $\vec{s}$-directional sides will be in $\partial E \setminus \partial F$ whereas the $\vec{t}$-directional sides are in $\partial F \setminus \partial E$. By making the $\vec{s}$-directional sides sufficiently long it will follow that $F$ will have less $\Phi$-energy than $E$, contradicting that a $\Phi(\cdot \, ; \R^{2})$-minimizing $E$ can have $\partial E$ containing more than one line.. One difficulty that makes the proof unnecessarily technical, is we need some bounded open set $A_{0}$ so that making this change on the rectangle above ensures that $E \Delta F$ is compactly supported in $A_{0}$. We do this by slightly fattening the rectangle we modify. 

More precisely, rescale and choose your origin so that $L_{i}$ is the line $\{x \in \R^{2}:  x \cdot \vec{t} = (-1)^{i}\}$ for $i \in \{1,2\}$.

For each $\sigma, \tau > 0$ define the rectangle
$$
R_{\sigma, \tau} = \{ x \in \R^{2} : -\sigma \le  x \cdot \vec{s} \le \sigma,  -\tau \le x \cdot \vec{t} \le \tau \}
$$

Define $a,b > 0$ so that $\max \{\| \vec{s} \| , \| - \vec{s} \|\} = a$ and $\min \{ \| \vec{t} \|, \| - \vec{t}\| \} = b$. Choose $\delta > 0$ so that $R_{1, 1 + \delta} \cap \partial E = R_{1, 1+\delta} \cap \left(L_{1} \cup L _{2} \right)$. That is, choose $\delta$ so that ``fattening'' $R$ vertically by a distance of $\delta$ does not meet any new pieces of $\partial E$. Fix $\rho > \frac{b}{a}$ and observe that
$$
\int_{(L_{1} \cup L_{2}) \cap \partial R_{\rho,1}} \| \nu_{E}\| \ge 4 \rho b > 4 a \ge \int_{\partial R_{\rho,1} \setminus \left(L_{1} \cup L_{2}\right)} \|\nu_{R}\|.
$$

Then, defining $F = E \setminus R_{\rho,1}$ or $F = E \cup R_{\rho,1}$ depending on whether or not $R_{\rho,1} \subset E$ it follows that $\Phi(F;R_{\rho+\delta,1+\delta}) < \Phi(E;R_{\rho+\delta; 1+\delta})$ contradicting the minimality of $E$ and hence verifying $\partial E$ is a single line, so that $E$ is a half-space.

\end{proof}

\bibliographystyle{alpha}
\bibdata{references}
\bibliography{references}
\end{document}